\documentclass[12pt, oneside,reqno]{amsart} 
\usepackage[utf8]{inputenc}
\usepackage{enumitem}
\usepackage{amsfonts}
\usepackage{amsmath,amsthm}
\usepackage{comment}
\usepackage{listings}
\usepackage{graphics}
\usepackage{bm}
\usepackage[inline]{asymptote}
\usepackage{float}

\usepackage{geometry}
\usepackage{graphicx,xcolor}
\graphicspath{{./figures/}}
\usepackage{hyperref,url}
\definecolor{darkblue}{rgb}{0,0,0.75}
\definecolor{darkred}{rgb}{0.75,0,0}
\definecolor{darkgreen}{rgb}{0,0.75,0}   
\hypersetup{colorlinks,
  filecolor=black,
  linkcolor=darkblue,
  citecolor=darkgreen,
  urlcolor=darkred,
  bookmarksopen=true}
\usepackage{amsthm}
\usepackage{graphicx}
\usepackage{amssymb}
\usepackage[section]{placeins}
\usepackage{cite}
\usepackage{appendix}
\usepackage{indentfirst}
\usepackage{extarrows}
\usepackage{tikz}

\numberwithin{equation}{section}
\newtheorem*{question}{Question}

\newtheorem{thm}{Theorem}[section]
\newtheorem{lem}[thm]{Lemma}
\newtheorem{prop}[thm]{Proposition}
\newtheorem{cor}[thm]{Corollary}
\theoremstyle{definition}
\newtheorem{defi}[thm]{Definition}
\newtheorem{rmk}[thm]{Remark}

\newtheorem{defn}[thm]{Definition}

\newcommand{\R}{\mathbb{R}}

\swapnumbers

\numberwithin{equation}{section}

\title[Noncompact CSF with convex projections]{On Asymptotically Conical Curve Shortening Flows with Convex Projections}
\author{Alexander Mramor}
\address{Department of Mathematics, University of Oklahoma, Norman, OK 73019, USA}
\date{}
\email{amramor@ou.edu}

\begin{document}

\begin{abstract} In this work we study the long term behavior of the curve shortening flow of asymptotically conical curves with convex projections in $\R^3$.
\end{abstract}

\maketitle

\section{Introduction}

The mean curvature flow and particularly the curve shortening flow are well understood in codimension one but significantly less is known in higher codimension, for one because the comparison principle fails to hold in general. With this in mind, conditions that aid in constraining the potential pathologies that occur like graphicality, symmetry, curvature pinching, or entropy bounds are desirable; as an incomplete but fairly broad list of past works consider \cite{AltschulerGrayson, Wang2002_GraphicMCF,  Wang2004MCF, Wang2001Mean, Andrews2010Mean, lynch2020highcodimensionmeancurvature, Naff2022, naff2022singularity,  vogiatzi2023singularity, Smoczyk2004_longtime, Colding2020Complexity, litzinger2023singularities, nguyen2026highcodimensioncurveshortening }.
$\medskip$

One interesting condition that seems to be newer and less studied, broadly fitting into the graphicality and pinching categories above, is the one-to-one convex projection condition. We phrase it here and throughout in terms of projection onto the $xy$-plane for concreteness; in the following, we let $P_{xy}:\mathbb{R}^n=\mathbb{R}^2\times\mathbb R^{n-2}\rightarrow\mathbb{R}^2$ be the orthogonal projection onto this plane and for a space curve $\gamma$, we let $P_{xy}|_\gamma:\gamma\rightarrow xy$-plane be its restriction to $\gamma$: 
\begin{defi}\label{conproj}
We say that a smooth curve $\gamma\subset\mathbb R^n$ has \emph{a one-to-one convex projection} onto the $xy$-plane if $P_{xy}|_\gamma$ is injective and the projection curve $P_{xy}(\gamma)$ is convex and boundaryless. 
\end{defi}
It turns out that this condition is preserved under the curve shortening flow/CSF (sometimes also refered to as sCSF to emphasize codim$(\gamma)>1$) by the Sturmian principle, and in an exciting recent work Qi Sun \cite{QSun_2025} showed the following generalization of Grayson's theorem \cite{Grayson1987}: 
\begin{thm}[Theorem 1.4 in \cite{QSun_2025}]
    If an embedded space curve $S^1 \simeq \gamma_0 \subset \R^N$ has a one-to-one convex projection onto the $xy$-plane, then its curve shortening flow $\gamma_t$ exists on a time interval $[0, T)$, $T < \infty$, and becomes asymptotically circular as $t \to T$. 
\end{thm}
See also the pioneering work of H{\"a}ttenschweiler \cite{Httenschweiler2015CurveSF} and Minar\v{c}\'{i}k and Bene\v{s} \cite{Minarcik2020LongTerm} on the convex projection condition. Inspired by this result and Polden's work \cite{Polden1991EvolvingCurves} on noncompact curve shortening flows in $\R^2$, in this article we consider the long term behavior of noncompact CSF with convex projections. It will also be desirable to impose some control on the spatial asymptotics of our space curves, which is the content of the following:
\begin{defi}
 Letting $L_1,L_2$ be two distinct non-vertical half-lines/rays, we denote by $\theta$ the angle between the rays $L_1,L_2$. We will say a (connected) space curve $\gamma$ is \textit{asymptotically conical} if its is $C^\infty$ asymptotic to two such rays $L_1, L_2$. 
\end{defi}
The definition of asymptotically conical we give above is perhaps somewhat nonstandard when the angle between $L_1$ and $L_2$ is zero, where is this case the asymptotic cone would be a ray with multiplicity two. The asymptotic conical condition is natural in many contexts in the mean curvature flow and gives that the behavior of the flow will be well controlled at spatial infinity. Our main result then is the following:
\begin{thm}\label{mainthm} Suppose the connected, smoothly embedded space curve $\gamma_0:\R\rightarrow\R^n$ is $C^\infty$-asymptotic to rays $L_1,L_2$ in the $xy$-plane and that $\gamma_0$ has a one-to-one convex projection onto the $xy$-plane. Then we have:
\begin{enumerate} 
\item There exists a solution $\gamma_t$ to the CSF out of $\gamma_0$ defined on the time interval $[0, \infty)$ and is unique amongst such immortal flows. Furthermore, $\gamma_t$ will continue to be a curve with convex projections asymptotic to $L_1$ and $L_2$.
\item There are sequences $t_i \to \infty$ for which in the $C_{loc}^\infty$ topology:
\begin{enumerate}
\item If $\theta = \pi$, $\gamma(\cdot,t_i)$ converges to a line whose link in $S^2$ agrees with the union of the links of $L_1$ and $L_2$. 
\item If $\theta\in(0,\pi)$, the rescaled CSF $\frac{\gamma(\cdot,t_i)}{\sqrt{2t+1}}$ converges to an expander of opening angle $\theta$. 
\item If $\theta = 0$, then after appropriate recenterings $\gamma(\cdot, t_i)$ will converge to a grim reaper.
\end{enumerate}
\end{enumerate}
\end{thm}

In fact item (2) above is true more generally for ramps and we can gain a slightly stronger conclusion -- see theorem \ref{expander_ramp_thm} below. The notion of ramp in the sense of Altschuler and Grayson \cite{AltschulerGrayson} and more is discussed in the preliminaries section, section \ref{prelim}; it's another useful notion of graphicality for curves in $\R^3$. The choice of recentering and rescalings indicated above are the most natural ones with the codimension one case in mind. 
$\medskip$

We also point out that in the statement above the notions of asymptotically conical and convex projection are ``coupled" in that the asymptotic rays lay in the same plane we project down to -- it seems natural to make some sort of assumption along these lines, although perhaps it can be weakened to the assumption that $L_1$ and $L_2$ are not perpendicular to the $xy$-plane. Combining with the work \cite{QSun_2025} we have the following updated picture of the curve shortening flow with convex projections in $\R^3$: 
\begin{cor} Suppose $\gamma \subset \R^3$ is an embedded asymptotically conical curve with convex projections in the coupled sense above, where $\gamma$ is asymptotically conical possibly to the empty cone. Then there are sequences of times $t_i \to \infty$ such that $\gamma_{t_i}$ converges either to a round point, a line, or after appropriate rescalings/recenterings an expander or grim reaper. 
\end{cor}
The statement above gives a qualitative picture of the potential phenomena in this setting much in line with Grayson's and Polden's results in the codimension one case, although there are still many interesting questions to consider. The most important questions/issues to address seem to be:
\begin{itemize}
    \item  In Theorem \ref{mainthm}, particularly for case 2(c), the sequences $t_i$ are judiciously picked and it would be interesting to know if the conclusions were true along any sequence of times $t_i \to \infty$. Imaginably there could still be ``bad" limits which don't fit into the list above. 
    \item In the argument for 2(c) we show that in the recentered limit one finds nontrivial grim reaper, although we don't show it itself is asymptotic to $L_1$ and $L_2$ and apriori it could be strictly narrower; it seems unlikely that this is the case however. See also remark \ref{widthrmk} below.
\end{itemize}

We conclude with a few words about the argument. In a nutshell, after establishing the long time existence of the flow we proceed to show that the flow becomes asymptotically planar, which is either automatically the case as in 2(a) above or by barrier arguments in the latter ones. After ensuring that one may take a limit of these flows we then may employ classification results in the codimension one case to gain our result. Case 2(c) is relatively harder to handle than the others because of the apparent lack of good a priori estimates along the flow; in this case we take limits of $\gamma_t$ as Brakke flows, and then argue that the limiting object is smooth using a geometric argument via the convex projection condition, at least along a certain choice of recentered flows. Afterwards its standard to verify the convergence is in fact smooth after potentially passing to a further subsequence using Brakke regularity theorem. 
$\medskip$

\textbf{Acknowledgments:} The author heartily thanks Qi Sun for his many valuable inputs on this work.
$\medskip$

\textbf{Statement on AI:} AI wasn't used to come up with the statements, arguments, or prose of this note, although it was used for proofreading, to double check some calculations, and to review the literature.

\section{Preliminaries}\label{prelim}

In this section we give a brief review of the mean curvature flow, including some facts which will be particularly important in the sequel. We start with defining the flow: 

    \begin{defn} 
        A (smooth) mean curvature flow (MCF) of embedded submanifolds in $\mathbb{R}^{m+n}$ is given by a manifold $M^n$ and a family $F: M \times I \to \mathbb{R}^{m}$ of embeddings satisfying
        \begin{equation}
            \frac{\partial F}{\partial t}\left(x,t\right) = \vec{H}
        \end{equation}
        where $I \subseteq \mathbb{R}$ is some nonempty interval and $\vec{H}$ is the mean curvature vector of the embedding at $(x,t)$. 
    \end{defn}
   When $m$, the dimension of $M$, is equal to $1$ the MCF is referred to as the \textbf{curve shortening flow}, abbreviated \textbf{CSF}. The flow is much less well understood when the codimension, $n$, is greater than one; one compelling geometric reason for this is because the avoidance principle fails to hold in high codimension as the figure below illustrates:
   \begin{figure}[H]
\hspace{1.5cm}\includegraphics[scale=.3]{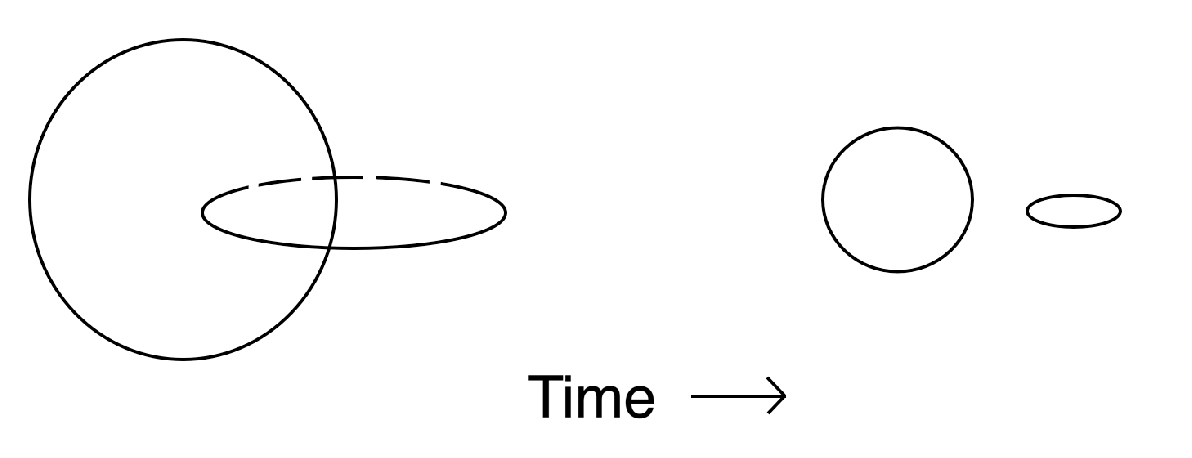}
\newline
\caption{Illustration of the evolution of two linked round circles laying in separate planes. They both shrink homothetically under the flow converging as sets to their centers, and before their flows become unlinked they must intersect each other.}
\end{figure}
However, it happens to be the case that the CSF will stay disjoint from the flow of convex domains, so that these can be used as barriers:

\begin{lem}\label{convbarrier} Let $\Gamma_t$ be a curve shortening flow and $M^2_t$ a convex solution to mean curvature flow so that $\Gamma$ is initially disjoint from the convex region $M$ bounds and which in later times, $t \in [0, T)$, $T \leq \infty$, might only intersect in a bounded region. Then, in fact, they remain disjoint for $t \in [0,T)$. 
\end{lem}

See lemma 2.3 in \cite{BourniLangfordLynch+2023+273+305} for a proof. As is well understood the mean curvature flow will often develop singularities, the phenomenon where the flow "pinches" at some points and time and the limiting set is no longer a smooth manifold. In most cases in this work though we will be able to avoid this situation, because in most situations the initial data we consider will in fact be \textbf{ramps}, which are a high codimension analogue of graphs for the CSF: 

\begin{defn}\label{ramp} Let $\gamma \subset \R^3$ be an embedded, arclength parameterized, curve with tangent vector $T$. Then $\gamma$ is said to be a ramp if $\langle T, V \rangle \geq 0$ for some vector $V$. 
\end{defn} 

The evolution equation for $\langle T, V \rangle$ is given by the following, where $\kappa$ is the geodesic curvature and $s$ is the arclength parameter.
\begin{equation}
\frac{\partial}{\partial t} \langle T, V \rangle = \frac{\partial^2}{\partial s^2} \langle T, V \rangle + \kappa^2  \langle T, V \rangle\,. 
\end{equation} 
To ensure that the ramp condition is preserved one would wish to use the maximum principle, but because we are considering noncompact flows throughout some additional care is required. Of course, there are noncompact maximum principles available but often in the sequel (in particular sections \ref{line})  we will have that $ \langle T, V \rangle $ is uniformly positive and stays by a barrier argument, discussed below, far away from the origin. Therefore, one may apply the maximum principle as in the compact case to see that this quantity stays positive and in particular its minimum doesn't decrease in many cases. The following evolution equation gives that lower bounds of $ \langle T, V \rangle$ along with initial upper bounds on $\kappa$ give bounds on $\kappa$ in later times.
\begin{equation} \label{ktv}
\frac{\partial}{\partial t} \frac{\kappa}{ \langle T, V \rangle} = \frac{\partial^2}{\partial s^2}\frac{\kappa}{ \langle T, V \rangle} + 2 \frac{2}{ \langle T, V \rangle} \frac{\partial}{\partial s} \langle T, V \rangle \frac{\partial}{\partial s}\frac{\kappa}{ \langle T, V \rangle} - \frac{\kappa}{ \langle T, V \rangle} \tau^2\,,
\end{equation} 
where $\tau$ here is the torsion.
As before from \eqref{ktv} it will often be the case $\frac{\kappa}{ \langle T, V \rangle}$ must be non increasing by the classical maximum principle (so without having to resort to noncompact ones). The relevance of ramps in this context is apparent with the following observation in hand:
\begin{prop}\label{isramp} If $\gamma: \R \to \R^n$ has convex projections and the angle $\theta$ between $L_1$ and $L_2$ is greater than zero, $\gamma$ is a ramp with respect to some $V \in \langle e_1, e_2 \rangle$.
\end{prop} 

This quality makes the study of the flow of $\gamma_t$ fairly straightforward as $t \to \infty$ when the opening angle $\theta > 0$. In the case $\theta = 0$ it's not clear that $\gamma$ is a ramp and in particular it isn't clear that $\kappa$ along $\gamma_t$ will be uniformly bounded so that we cannot take smooth limits as $t \to \infty$. Instead we will take limits as Brakke flows, which we next define: 
\begin{defn}[Brakke flow]
        A (k-dimensional integral) Brakke flow is a family of Radon measures $\mu_t$ such that, on an interval $I \subseteq \mathbb{R}$:

        \begin{enumerate}
            \item For almost every $t \in I$ there exists an integral $k$-dimensional varifold $V(t)$ with $\mu_t = \mu_{V(t)}$ so that $V(t)$ has locally bounded first variation and has mean curvature vector $\vec{H}$ orthogonal to Tan$(V(t), \cdot)$ a.e. 
            \item For a bounded interval $[t_1, t_2] \subset I$ and any compact set $K \subset \mathbb{R}^{n}$,
            \begin{equation}
                \int_{t_1}^{t_2} \int_K (1 + |\vec{H}|^2) d\mu_t dt < \infty.
            \end{equation}
            \item (Brakke's inequality) If $[t_1, t_2] \subset I$ and $\phi \in C^1_c(\mathbb{R}^{n} \times [t_1, t_2])$ with $\phi \geq 0$, then
            \begin{equation}
                \int_{V(t_2)} \phi(\cdot, t_2) d\mu_{t_2}  \leq \int_{V(t_1)} \phi(\cdot, t_1) d\mu_{t_1} + \int_{t_1}^{t_2} \int_{V(t)} -\phi |\vec{H}|^2 +  \langle \nabla \phi, \vec{H} \rangle + \frac{d\phi}{dt} d\mu_t dt.
            \end{equation} 
        \end{enumerate}
    \end{defn}
     The study of Brakke flows itself is quite rich, and a comprehensive introduction to Brakke flows can be found in, say, \cite{Ton}. Most relevantly for us is \textbf{Brakke's compactness theorem}, which says that for a sequence of Brakke flows with uniform area bounds on parabolic cylinders, one may exact a subsequence which converges to another Brakke flow. This convergence will be in the sense of Radon measures for all times and as varifolds for almost all times. A second important result is \textbf{Brakke's regularity theorem}, which says that Brakke flows with density bounds sufficiently close to 1 in a backwards parabolic neighborhood will be smooth with bounded curvature in a smaller neighborhood. This will play an important role in the proof of item 2(c) of Theorem \ref{mainthm} where apriori we will only take limits in the sense of varifolds.

\section{Proof of item (1) of Theorem \ref{mainthm}}\label{exist_section}

Throughout this section we suppose that $\gamma$ is a smooth embedded curve in $\R^3$ (or just as well $\R^n$) which is asymptotic to two rays/half lines, $L_1$ and $ L_2$. Our first goal is to merely establish short time existence of the flow; towards this we recall the local Lipschitz condition of M.T. Wang \cite{Wang2004MCF}:

\begin{defi} Given any positive $K < 1$, a compact n-dimensional submanifold $\Sigma$ of $\R^{n+m}$ is said to satisfy the K local Lipschitz condition if there exists a $r_0 > 0$ such that $\Sigma\cap B(q,r_0)$ for each $q \in \Sigma$ can be written as the graph of a vector valued Lipschitz function $f_q$ over an n-dimensional aﬃne space $L_q$ through q with $\frac{1}{\sqrt{ det(I + (df_q)^T df_q)}} > K$.
\end{defi} 

Since $\gamma$ asymptotically converges to $L_1, L_2$ for any $\epsilon > 0$ there exists an $R_0(\epsilon) >> 0$ so that for any $q \in \gamma \cap B(0, R_0)^c$, $\gamma$ is locally Lipschitz in the sense above with $r_0 = 1, K =1-  \epsilon$. In particular this implies:

\begin{lem} For any $K$ close to 1, $\gamma$ satisfies the $K$ local Lipshitz condition. 
\end{lem} 

As discussed in section 5 and corollary 4.1 of Wang, if $K$ is sufficiently close to 1 (depending on $m,n$) there is a lower bound on the time of existence for $\gamma$, as well as curvature (and higher derivative) bounds, in the case that $\gamma$ is \textit{compact}. To reduce to this case, we cap off $\gamma$ with a semicircular arc in $S(0,R_i)$ along a sequence $R_0 < R_i \to \infty$ to produce compact curves $\gamma^i$. Its easy to see this can be done in such a way so that each curve $\gamma^i$ is uniformly K locally Lipschitz for a value of K sufficiently close to 1 that local curvature estimates hold. 
$\medskip$

Since along the sequence $R_i \to \infty$, $\gamma^i$ converge on compacta to $\gamma$ we can run the CSF on each $\gamma^i$ for a definite positive time by the aforementioned curvature bounds. Now we can take a converging subsequence ala Arzela-Ascoli and a diagonalization argument to define a smooth flow $\gamma_t$ out of $\gamma$. Indeed by a barrier argument employing osculating spheres and lemma \ref{convbarrier} above to the approximating flows $\gamma^i_t$, the flow $\gamma_t$ will converge back to $\gamma$ as $t \to 0$ (see also proposition 2.1 in \cite{Wang2004MCF}). By the uniqueness theorem of Chen and Yin \cite{ChenYin2007_MCF_pseudolocality} it will be unique amongst smooth flows over subintervals $I_s = [0, s] \subset [0,T)$. Summing up our discussion so far: 

\begin{prop} For any smoothly embedded space curve $\gamma$ asymptotic to rays $L_1, L_2$, there exists a smooth CSF $\gamma_t$ out of $\gamma$ for some time period $[0, T)$ and it is unique among all such smooth flows over subintervals $I_s = [0,s] \subset [0,T)$.
\end{prop}

Next, we discuss the spatial asymptotics of the $\gamma_t$ at some finite forward time:

\begin{lem}\label{halfline_asymp} Suppose that $\gamma_t$ is a noncompact curve shortening flow defined on $I= [0, T)$, $T < \infty$ such that $\gamma_0$ is smoothly asymptotic to rays $L_1$, $L_2$. Then for any $t \in I$, $\gamma_t$ is also smoothly asymptotic to $L_1$ and $L_2$.
\end{lem}

\begin{proof}
  Since $L_1 \neq L_2$ its easy arrange convex barriers, using lemma \ref{convbarrier}, so that $\gamma_t$ will be $C_0$ asymptotic to $L_1$ and $L_2$ for all $t$, of course with potentially worsening rate of decay as $t$ increases. In fact, using that the initial data is initially smoothly asymptotic to $L_1, L_2$ its not hard to show that the flow will then remain smoothly asymptotic to $L_1$ and $L_2$, because the initial smoothness along with the $C^0$ estimates gives control on the area ratios of $\gamma_t$ when near $L_1$ and $L_2$ for small times. Brakke regularity and higher order estimates can then be invoked to propagate the smooth convergence to larger times, as in proposition 2.6 of \cite{MP_flowtoroundpoint}.
\end{proof}

Using that $L_1, L_2$ lay in the $xy$-plane a corollary of this, by the Sturmian theorem \cite{angenent1988zero} proceeding essentially as in section 2 of \cite{QSun_2024}, is the following:
 \begin{cor}\label{stayconvex+ramp} Suppose that additionally $\gamma$ has 1-1 convex projections, in the sense of definition \ref{conproj}. Then for all $t \in [0, T)$ $\gamma_t$ will have 1-1 convex projections with uniformly bounded slope, and in particular will remain a ramp in the case $\theta > 0$.
 \end{cor}
 
 What remains is to show that the flow exists smoothly for all times. One approach could possibly be to extend work in \cite{QSun_2025} to this setting which it seems would most nontrivially involve modifying some barrier arguments, but instead we will take a case by case approach which is technically simpler: 
 \begin{lem} Suppose that $\gamma_t$ is the flow out of $\gamma_0$ asymptotic to half lines $L_1, L_2$ as above defined on a maximal time interval $[0, T)$. Then $T = \infty$. \end{lem} 
 \begin{proof}
In the case $\theta > 0$, its easy to see there is a vector $V$ in the $xy$-plane such that $\langle T, V \rangle > c >0$ for all $x \in \gamma_0$, so $\gamma_0$ will be a ramp with respect to $V$. Since this quantity is uniformly bounded below and $\gamma_t$ will remain asymptotic to $L_1$ and $L_2$, $\langle T, V \rangle$ will be positive and bounded from below for as long as the flow exists. With this in mind, since $\gamma_t$ for any $t > 0$ is smoothly asymptotic to $L_1$ and $L_2$ the maximum of $\frac{\kappa}{\langle T, V \rangle}$ is always achieved by some point on $\gamma_t$, and so from the evolution equation \ref{ktv} we see that it is nonincreasing; in particular when $\theta > 0$ there is some constant $C = C(\gamma_0) >0$  so that $\kappa < C$ along $\gamma_t$, implying $T = \infty$ in this case. 
$\medskip$

Next we consider the case $\theta  = 0$ and suppose that a singularity occurs at the spacetime point $(p^*, t^*)$. Where $P_{xy}(\gamma_{t^*})$ is the projection of this curve, one can see that a tangent flow performed about $P_{xy}(\gamma_{t^*})$ must be a multiplicity 2 line, using that tangent flows of CSF $\gamma_t$ of any codimension are planar by \cite{Altschuler1991SingularitiesOT}. Since $P_{xy}(\gamma_{t^*})$ is asymptotic to $P_{xy}(L_1) \neq  P_{xy}(L_2)$, $P_{xy}(\gamma_{t^*})$ cannot bound a convex domain which is a contradiction since it is a limit of convex curves. 
\end{proof}

To recap we first established short time existence (in the class of bounded curvature) for flows $\gamma_t$ from space curves $\gamma$ asymptotic to half lines/rays. Then we showed that the property of being asymptotic to half lines was preserved at any finite forward time, which implied that the convex projection property would continue to hold if it did initially. With this in hand we concluded that the flow $\gamma_t$ must continue smoothly for all time. Considering subintervals $I_n = [0, n] \subset [0, \infty)$, the uniqueness result of Chen and Yin \cite{ChenYin2007_MCF_pseudolocality} gives it is the unique such smooth flow.

\section{Proof of items 2(a) and 2(b) of theorem \ref{mainthm}}\label{line}\label{exapander}

We group both of these cases together because in this setting $\gamma$ will be a ramp, which greatly simplifies matters.
$\medskip$

The proof of item 2(a) is quite simple, one reason morally speaking because no rescaling or recentering is done; indeed in this context, because the space curve shouldn't flow to spatial infinity or develop a singularity translations or dilations shouldn't be necessary to capture the long term behavior of the flow. Now, since $P_{xy}(\gamma)$ is convex and $\theta = \pi$ $P_{xy}(\gamma)$ is a line $\ell$ so that already we know $\gamma$ (and hence $\gamma_t)$ is contained in a plane; in fact by definition of convex projection it will be a bounded graph over $\ell$. Classical results of Ecker and Huisken \cite{ecker1989mean} then imply that it will flow to a line as $t \to \infty$. 
$\medskip$

For the remainder of this section we discuss item 2(b). In fact, we will show a somewhat more general result, which in a sense generalizes item 2(a):
\begin{thm}\label{expander_ramp_thm} Suppose $\gamma$ is an asymptotically conical space curve asymptotic to $L_1$ and $L_2$ where the angle $0< \theta \leq \pi$, and furthermore suppose that $\gamma$ is a ramp. Then the rescaled flow $\frac{\gamma_t}{\sqrt{2t + 1}}$ will subsequentially converge to an expander asymptotic to $L_1$ and $L_2$. 
\end{thm}

As for the case of item 2(a), its easy to show the the curvature of $\gamma_t$ will be bounded by some uniform constant for all times, but in this case one expects (and indeed, it must occur) that the flow clears out, in that for any bounded domain $D$ there will be some time $T(D)$ for which $\gamma_t \cap D = \emptyset$ for $t > T$. To proceed we rescale the flow, following Ecker and Huisken:
\begin{defi}[Following Ecker-Huisken \cite{ecker1989mean}]
\label{definition of the rescaled CSF}
We define the \emph{rescaled CSF} to be
    \[
    \Gamma(u,\tau):=\frac{\gamma(u,t)}{\sqrt{2t+1}},\quad 
    \tau:=\frac{1}{2}\log(2t+1).
    \]

\end{defi}
In their seminal paper Ecker and Huisken then show that codimension 1 graphs $M$ satisfying some growth conditions will, after applying the rescalings above, converge to an expander. An important ingredient in their argument is their estimate (see proposition 4.4 in \cite{ecker1989mean}): 
\begin{equation} 
t^{m+1} | \nabla^m A|^2 \leq C(m)
\end{equation}
where the constant $C(m)$ depends also on $n$ and the initial gradient bound (roughly) of the graph $M$. With this one can see after rescaling the flow by $\frac{1}{\sqrt{2t+1}}$, the rescaled curvature will remain bounded so that one could apply smooth compactness results. Note that taking $m = 0$ as $t \to \infty$, $|A|^2 \to 0$ along the flow $M_t$ of $M$, which of course is a stronger conclusion than what one obtains from applying the maximum principle to equation \ref{ktv}. Since $|D\gamma|$ is scaling invariant we can at least say the following using these estimates however, which will be important below:
\begin{lem} $|D\gamma|$ is uniformly bounded, in $x$ and $t$, on the rescaled CSF. 
\end{lem}

In particular, these bounds are useful in appling Nash--Moser estimates. In the following, we use $(x,y,z)$ to denote the coordinates of the rescaled CSF $\Gamma_t$. Recalling by corollary \ref{stayconvex+ramp} that $\gamma_t$ and hence $\Gamma_t$ are ramps where without loss of generality $V = e_1$, we see that both $y$ and $z$ are graphical in $x$. The evolution equation $y$ satisfies in the following:
\begin{equation}
\label{eq: rescaled graphical equation to derive estimates}
y_\tau=
ay_{xx}+xy_x-y, \quad a=\frac{1}{1+y_x^2+z_x^2}.
\end{equation}
Where the evolution equation for $z$ is similar. With this setup one can show the following, which suffices to take convergent subsequences of $\Gamma_t$ locally smoothly.
\begin{prop}
\label{curvature bound locally}
For each fixed $M<+\infty$, for $|x|\leq M$ and $\tau\in[0,+\infty)$, the rescaled curvature is uniformly bounded by some constant depending on $M$ and the global gradient bound.
\end{prop}
As mentioned this follows via Nash--Moser estimates and is discussed more at the end of this section. Continuing on, because $\gamma$ is asymptotic to $L_1, L_2$ it has bounded entropy, in the sense of Colding and Minicozzi \cite{ColdingMinicozzi2012}, implying local area bounds. With this and the derivative estimates in hand, for any sequence $\tau_i \to \infty$ we can extract a subsequence $\tau_{i_j}$ such that $\Gamma_{t - \tau_{i_j}}$ converges as a mean curvature flow on compacta; applying a diagonal argument denote the limit by $\Pi_t$, $t \in [0, \infty)$, which has the following properties:

\begin{lem}\label{lim_props} The limit curve $\Pi_t$ is a nonflat planar CSF, asymptotically conical with opening angle $\theta$ in $C^0$ topology.
\end{lem}
\begin{proof} Since $\gamma$ is asymptotic to $L_1, L_2$ it lays in a slab and will for all forward times, using lemma \ref{convbarrier} with appropriate pancake barriers \cite{blt_pancakes} running parallel along the sides of the slab. Denoting by $P_1, P_2$ the parallel planes bounding this slab, under the expander rescaling as $\tau \to \infty$ these planes are brought towards the origin, implying $\Pi_t$ is planar, and say contained in the plane $P_3$ -- of course because $L_1$, $L_2$ lay in the $xy$-plane $P_3$ must be this plane as well. Using appropriate planes parallel to $L_1$ and $L_2$ which intersect $P_3$ orthogonally, one can similarly see that $\Pi$ is bounded by a wedge in $P_3$ of angle $\theta$. To see its asymptotically conical with opening angle $\theta$ (and not narrower), we consider convex cylinders $\Phi$, i.e. convex surfaces given by the form $\rho(s) \times \R$ where $\rho$ is a convex curve asymptotic to $L_1$ and $L_2$ which don't intersect $\Gamma_t$. By \cite{stavrou1998selfsimilar}, under the rescaled flow $\rho_t$ will flow to an expander, and hence under the rescaled flow $\Phi_t$ will flow to a cylinder over an expander of angle $\theta$. By lemma \ref{convbarrier} and an approximation argument by closed surfaces similar to the existence and uniqueness argument in section \ref{exist_section} one can see $\Phi_t$ itself serves as a barrier to $\gamma_t$ and hence its rescaling to $\Gamma_\tau$, implying the claim on $\Pi_t$. 
\end{proof}
Now, if we apply the rescaled flow to $\Pi_0$ we see that by work of Stavrou \cite{stavrou1998selfsimilar} it must flow to an expander of opening angle $\theta$; since $\Pi_t$ arose as the limit of $\Gamma_{t - \tau_{i_j}}$ we see that there is a sequence $\tau_i \to \infty$ for which $\Gamma_{\tau_i}$ converges to a planar expander of opening angle $\theta$, as claimed, completing the proof of Theorem \ref{expander_ramp_thm} and hence item 2(b) of Theorem \ref{mainthm} modulo the rescaled curvature bounds:

\begin{proof}[Proof of proposition \ref{curvature bound locally}]

We show $|y_{xx}|$ is bounded, $|z_{xx}|$ is similar. Using that $\gamma$ is asymptotic to $L_1, L_2$ and $\theta > 0$ we have that $\gamma$ is a ramp with $y_x, z_x$ bounded. By De Giorgi-Nash-Moser type estimates, because $y_x$ and $z_x$ are bounded $y_x$ is Hölder continuous, see for example \cite{ladyzhenskaia1968linear}. Similarly $z_x$ is Hölder continuous giving that $a$ is Hölder continuous. With this in hand considering equation \eqref{eq: rescaled graphical equation to derive estimates} by Schauder estimates,
\[
\|y_{xx}\|_{L^\infty(Q_{M})}\leq C\|y\|_{L^\infty(Q_{2M})},
\]
One can see $|y|$ is uniformly bounded for all time for $|x|\leq M$ for fixed $M$ via a barrier argument similar to the proof of lemma \ref{lim_props}. \end{proof}

\section{Proof of item 2(c) of theorem \ref{mainthm}}

We next suppose that $\gamma$ is asymptotic to two parallel lines $L_1, L_2$, which we normalize to be distance 1 apart. In the following it will be helpful to fix a reference plane with which to measure height, where we recall $L_1$ and $L_2$ both lay in the $xy$ plane:
\begin{defi}
Denote by $Q$ a fixed plane perpendicular to the $xy$ plane and the slab bounded by $L_1$ and $L_2$ disjoint from $\gamma$, and $d_Q$ the signed distance from $Q$ so that for $x \in \gamma$ we have $d_Q(x) > 0$.
\end{defi}
Because $\gamma$ is connected its easy to see such a plane exists. By the avoidance principle, $d_Q > 0$ on $\gamma_t$ is preserved along the flow. By the assumption that $Q$ is not parallel to $L_1$ or $L_2$ there is (at least one) global minimum for $d_Q$ along $\gamma$, and this continues to hold true for $\gamma_t$ for any $t > 0$. For a given $t$, denote the minimum value obtained by $d_Q$ along $\gamma_t$ by $m_t$. Guided by the ansatz that such flows should be modeled asymptotically on grim reapers, we consider the family of flows $\gamma^i_t = \gamma^i_{t - t_i} - p_i$ where $d_Q(p_i) = m_{t_i}$ and $t_i$ is a sequence $t_i \to \infty$ to be determined; we wish to take a limit of these flows. 
$\medskip$

An important complication in this case is that it doesn't seem clear that $\gamma$ must be a ramp however, and in particular it is not clear we have control on $\kappa$. With this in mind to proceed we instead take weak limits. In the following, we denote by $\mathcal{D}$ the slab in $P$ bounded by $L_1$ and $L_2$:
\begin{lem}\label{reaper_in_plane} Suppose that $\gamma^i_t = \gamma^i_{t - t_i} - p_i$ where $d_Q(p_i) = m_{t_i}$ achieves the minimum of $d_Q$ on $\gamma_t$ and $t_i \to \infty$. Then as Brakke flows $\gamma^i_t$ subsequentially converge to a Brakke flow $\Gamma_t, 0 \leq t < \infty$ for which 
\begin{enumerate}
    \item The support of $\Gamma_t$ is contained in the plane $P$ containing $L_1$ and $L_2$.
    \item The support of $\Gamma_t$ bounds a convex set $D_t \subset \mathcal{D}$. 
\end{enumerate}
\end{lem}
\begin{proof}
    Because $\gamma$ is asymptotic to $L_1, L_2$ it has bounded entropy, in the sense of Colding and Minicozzi, implying local area bounds. Thus, we may apply Brakke compactness theorem to find a limiting Brakke flow $\Gamma_t$ after potentially passing to a subsequence. 
    $\medskip$
    
     Since $P_{xy}(\gamma)$ is convex and $\gamma$ is asymptotic to $L_1$ and $L_2$ $P_{xy}(\gamma)$ must be contained in $\mathcal{D}$, and this continues to hold by lemma \ref{convbarrier} using very large spheres touching the cylinder over $\mathcal{D}$ from outside as barriers.  Furthermore because as graphs over their projections $\gamma$ and hence $\gamma_t$ have bounded slope by the Sturmian principle and the asymptotic conical assumption, there is a uniform constant $s> 0$ so that $\tilde\kappa < s \gamma_t'' \cdot \vec{u}$, where $\tilde\kappa$ is the curvature along $P_{xy}(\gamma_t)$ and $\vec{u}$ is the unit normal of it lifted to $\gamma_t$. Indeed, by lemma 3.2 of \cite{QSun_2024} its easy to see we have $(x_s^2 + y_s^2)  \tilde{\kappa} = P_{xy}(\gamma''_t) \cdot \vec{u} = \gamma''_t \cdot \vec{u}$ where $x_s^2 + y_s^2$ on $\gamma_t$ is uniformly positive since $\gamma_t$ has bounded slope over the $xy$-plane, so we see there is an $0< s < \infty$ for which $\tilde\kappa < s \gamma_t'' \cdot \vec{u}$.
    $\medskip$
    
    As a consequence it's easy to see that CSFs in the $xy$-plane with speed scaled by $s$ and laying on the ``outside" of $P_{xy}(\gamma_t)$, that is in the nonconvex domain bounded by $P_{xy}(\gamma_t)$, serve as barriers to the projected flow (so long as there are distance minimizing pairs of points). In particular by using grim reapers of width strictly greater than 1 with speed scaled by $s$ as barriers to $P_{xy}(\gamma_t)$ we see that there is a constant $c = c(s) > 0$ so that $ct + m_0 \leq m_t $. 
     $\medskip$

     \begin{figure}
    \centering
    \includegraphics[width=0.75\linewidth]{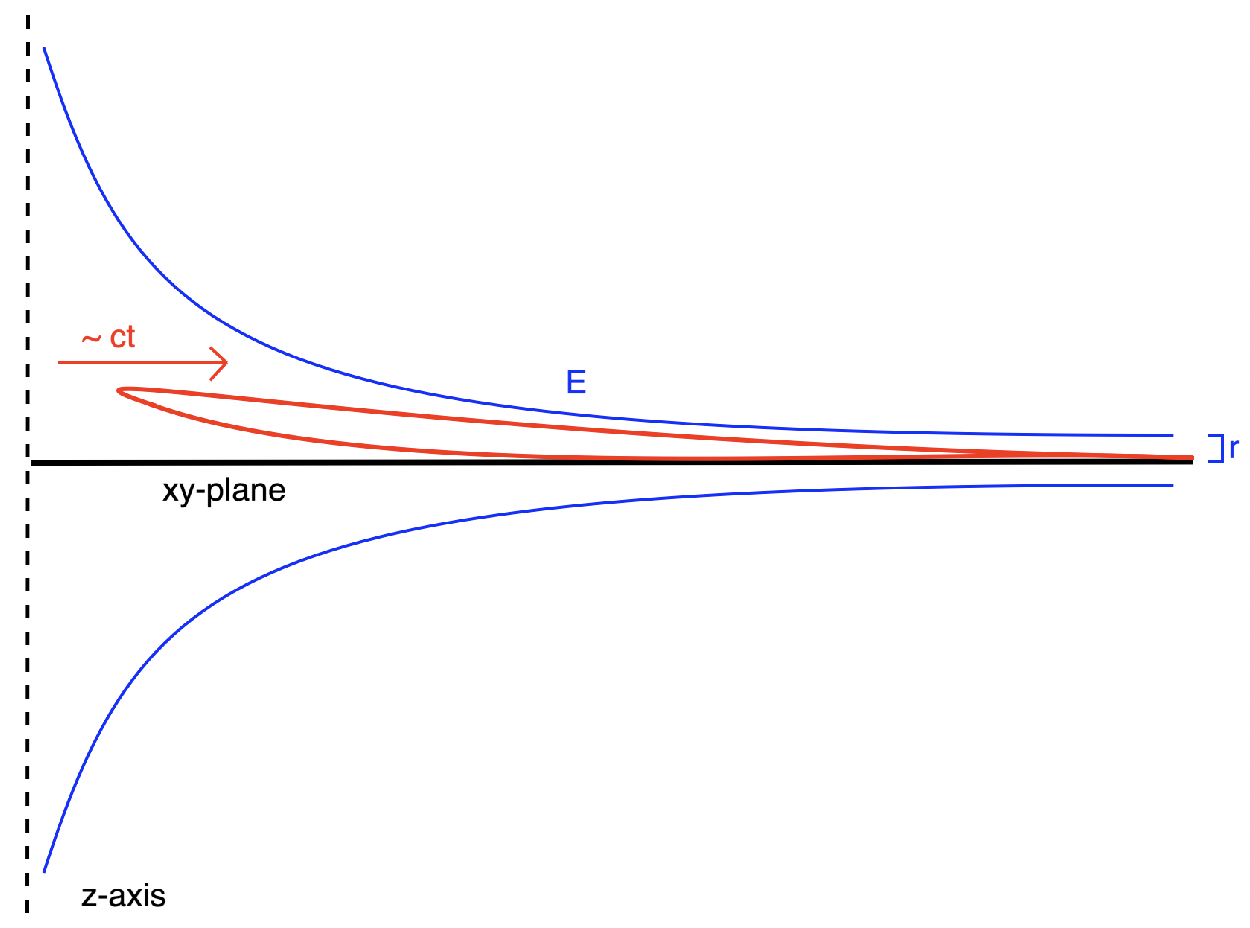}
    \caption{A side view of the barriers in lemma \ref{reaper_in_plane}, where the blue curves represent the translated expanders $E$ and its reflection across the $y$-axis and the red curve denotes $\gamma_t$ (where, from this perspective it would appear to self intersect since $L_1, L_2$ are both in the $xy$-plane). Roughly speaking, the tip of $\gamma_t$ in the figure moves to the right sufficiently faster than the the flow of $E$ separates from the $y$-axis.}
    \label{fig:squeeze}
\end{figure}

On the other hand, we may use a cylinder $\widetilde{E}$ over an appropriately translated expander $E$ asymptotic to the $y$-axis, of which there are many depending on opening angle, and its reflection across the $y$-axis as barriers to ``squeeze" $\gamma_t$ as illustrated in figure \ref{fig:squeeze}. In the following, we first suppose for simplicity that $\gamma_0$ lays completely in the $xy$-plane sufficiently far from the origin. Up to rotation we may $\mathcal{D}$ is parallel to the $y$-axis and we suppose $E$ is an expander in the $yz$-plane with basepoint $(y_0, z_0)$, arranged so that $\widetilde{E}$ is disjoint from $\gamma_0$; it and its cylinder $\widetilde{E}$ will be convex and one can see considering compact approximations that it will serve as a barrier to $\gamma_t$ by lemma \ref{convbarrier}. Then if $E$ is asymptotic to the $y$-axis with rate $r(|(y - y_0, z - z_0)|)$ its flow $E_t$ will be asymptotic to the $y$-axis with rate $\sqrt{2t + 1}r(|(y - y_0, z - z_0)|/\sqrt{2t + 1})$ by the expander property. Note from section 4 of \cite{Ding2020Minimal} that for our expander $E \subset \R^2$ we can arrange that $r(d) = O(e^{-c_0d^2}/d)$ for some $c_0 > 0$, where $d$ is the distance to basepoint. Inserting in the lower bound $ct + m_0 - (|y_0| + |z_0| + \sup |z(\gamma_0(s))|) \leq m_t - (|y_0| + |z_0|+\sup |z(\gamma_0(s))|)$ (the inequality from above shifted by $|y_0| + |z_0| + \sup |z(\gamma_0(s))|$ and crudely a bound for the distance between the ``tip point" $p_t$ and the cylinder/line over the basepoint) we see for points on $\widetilde{E}_t$ laying over $\gamma_t$, in terms of the projection map $P_{xy}$, that $\sqrt{2t + 1}r(|(y - y_0, z - z_0)|/\sqrt{2t + 1}) \sim e^{-c_1t} \to 0$ as $t \to \infty$ for an appropriate $c_1 > 0$, implying that the support of $\Gamma_t$ is contained in the $xy$-plane. In the general case that $\gamma_0$ is only asymptotic to the $xy$-plane we can repeat the argument above using a family of expanders asymptotic to successively thinner slabs about the $xy$-plane farther and farther away from the origin. 
$\medskip$
    
    Lastly, because each of the $\gamma^i_t$ have 1-1 convex projections one can see that the support of $\Gamma_t$ itself bounds a convex set $D_t$ for all $t \geq 0$. 
\end{proof}

Certainly $\Gamma_0$ will be nonempty by the choice of recenterings, but it could be the case the support of $\Gamma_0$ is a ray which immediately disappears under the flow. Our next task then is to show that we may choose a sequence $t_i \to \infty$ so that $D_t$ has nonempty interior.
\begin{lem}\label{nondegen} There exists a sequence of times $t_i \to \infty$ so that the convex sets $D_t$ have nonempty interior for all $t \geq 0$.
\end{lem} 
\begin{proof}  Recall that we supposed for the sake of concreteness the rays $L_1$ and $L_2$ are distance one apart at the start of this section. First we show there is an appropriate sequence $t_i \to \infty$ so that the conclusion holds for $t = 0$. If $D_0$ has empty interior then it must be a halfline by convexity, which since the projected curves are convex we see that for any $\epsilon > 0$ one has $P_{xy}(\gamma^i_0) \cap B(0,100)$ is contained in a slab of width $\epsilon$ for $i$ sufficiently large. Suppose for the sake of contradiction that this is the case for any sequence $t_i \to \infty$ which, for a fixed choice of $\epsilon$, implies there is a time $T(\epsilon) \gg 0$ so that $P_{xy}(\gamma_t - p_t) \cap B(0,100)$ is contained in a slab of width $\epsilon$, where $p_t$ is a distance minimizer to the reference plane $Q$ from above on $\gamma$ at time $t$. 
$\medskip$

To continue, we recall from the previous argument that there is a constant $s> 0$ so that $\tilde\kappa < s \gamma_t'' \cdot \vec{u}$, where $\tilde\kappa$ is the curvature along $P_{xy}(\gamma_t)$ and $\vec{u}$ is the unit normal on $P_{xy}(\gamma_t)$ lifted to $\gamma_t$. For a given $0 < \epsilon \ll1$ we use, similar to the previous proof, grim reapers of width $\sqrt{\epsilon}$ with speed scaled by $s$ as barriers for $\gamma_t$ for $t > T$ within $B(p_t, 100)$ -- of course these grim reapers should intersect $\gamma_t$ but from the previous reductions they must do so outside the ball $B(p_t, 100)$. For $\epsilon$ sufficiently small and a given time $t_0 > T$ one can arrange the distance between the tips of these grim reapers and $P_{xy}(p_i)$ is less than $\sqrt{\epsilon}$, which implies there is some constant $C > 0$ so that $\frac{dm}{dt}(t_i) > C/\sqrt{\epsilon}$ for $t > T$. By taking $T$ larger/$\epsilon$ smaller, this can be arranged to be as large as we wish; see figure \ref{fig:enterprise}.
$\medskip$

On the other hand, denote by $\mathcal{G}$ a grim reaper asymptotic in the slab $\mathcal{D}$ bounded by $L_1$ and $L_2$ of width $1/2 < w < 1$. By appropriately translating it, $\mathcal{G}$ can be arranged to be disjoint from $\gamma_0$. Noting that for a fixed $t > T$ $\gamma_t$ is asymptotic to $L_1$ and $L_2$ we obtain a contradiction via lemma \ref{convbarrier} using cylinders over $\mathcal{G}$ (more precisely, compact convex approximations of them) for $\epsilon$ sufficiently small because these grim reaper cylinders translate at a speed bounded by $1/2\pi$. So, it must be the case that there is some $\epsilon' > 0$ and some sequence of times $t_i \to \infty$ for which $P_{xy}(\gamma^i_0) \cap B(0,100)$ is \textit{not} contained in a slab of width $\epsilon'$, giving the claim for $t = 0$. By the convexity assumption using an appropriately sized grim reaper as an inner barrier one can see that $D_t$ is nondegenerate for all $t > 0$ as well. \end{proof}

\begin{figure}
    \centering
    \includegraphics[width=0.7\linewidth]{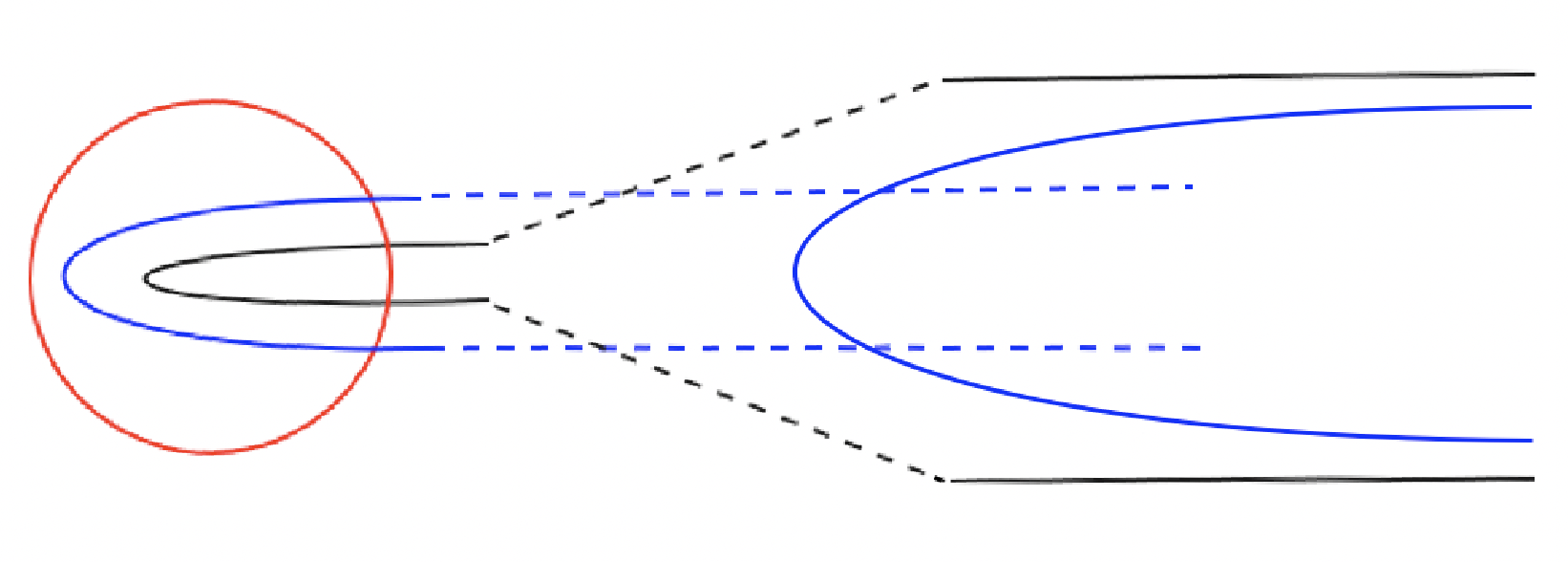}
    \caption{A schematic of the barriers in lemma \ref{nondegen}, where the black curve represents $P_{xy}(\gamma)$ and the blue curves the various grim reapers with one centered over the ball $B(p_t, 100)$ outlined in red. The smaller grim reaper, along with the convexity of $P_{xy}(\gamma)$, insure that that $\gamma_t$ will "translate" very quickly, which is at odds with the slower moving larger grim reaper to the right. Note that proportionally the ``tip" of $P_{xy}(\gamma)$ will be much thinner than the outer reaper barrier for $\epsilon \ll 1$.}
    \label{fig:enterprise}
\end{figure}

We next claim regularity of this limit. In the following, we denote by $\mathcal{S}$ the singular set of $\Gamma_t, t > 0$:
\begin{lem} For $t > 0$ $\Gamma_t$ is smooth, or in other words $\mathcal{S} = \emptyset$ and consequently the convergence $\gamma_t^i \to \Gamma_t$ is smooth after potentially passing to a subsequence.  
\end{lem}
\begin{proof}
Suppose for the sake of contradiction there is a singular point $p_s$ encountered along the flow. By the local area bounds we may consider the tangent flow at a singular point $p_s$ obtaining a nontrivial self shrinker $\Sigma$. Note that since $\Gamma_t$ is an integral Brakke flow, the density of every point in the support of $\Sigma$ is at least one. This self shrinker asymptotically converges to a cone $C \subset \mathbb{R}^2$, at least in a weak sense (see proposition 2.2 in \cite{Lu}) which in this case is asymptotic to a (potentially empty) collection of rays possibly with multiplicity; since the limit shrinker $\Sigma$ must also have convex and noncompact support its not hard to see that $\Sigma$ is in fact precisely a line of multiplicity 2. From this we gain a contradiction by the convexity of the whole set $D_t$ and it's nondegeneracy from lemma \ref{nondegen} implying our claim of regularity. To see that the convergence of $\gamma_t^i \to \Gamma_t$ is smooth we note that because $D_t$ is nondegenerate for each $t$ and the space of integral varifolds is closed that the multiplicity of convergence must be one, so the Brakke regularity theorem implies that along the sequence we have local curvature bounds with which we may extract a smoothly converging subsequence ala Arzela-Ascoli. 
\end{proof}

Thus $\Gamma_t$ is a smooth convex flow. For each $t >0$ $\Gamma_t$ is complete, nonflat, and trapped in a slab, so must be asymptotic to two distinct parallel rays and consequently will flow to a grim reaper by \cite{Polden1991EvolvingCurves} or \cite{choi2021convergence}, implying item 2(c) in theorem \ref{mainthm}.

\begin{rmk}\label{widthrmk}
    As pointed out in the introduction, note that we don't show here the asymptotic grim reaper is itself asymptotic to $L_1$ and $L_2$. It seems one should be able to show that the constant $s$ from lemmas \ref{reaper_in_plane} and \ref{nondegen} should tend to $1$ along $\gamma_t$ as $t \to \infty$, so that for $t$ large $P_{xy}(\gamma_t)$ is approximately a curve shortening flow in the plane. Supposing this is the case, speculatively perhaps one can refine the arguments above or extend the arguments of \cite{choi2021convergence} to such flows to see they remain asymptotic to $L_1$ and $L_2$ in the limit. \end{rmk}

\bibliographystyle{plain}
\bibliography{main}

\begin{thebibliography}{10}

\bibitem{Altschuler1991SingularitiesOT}
Steven~J. Altschuler.
\newblock Singularities of the curve shrinking flow for space curves.
\newblock {\em Journal of Differential Geometry}, 34:491--514, 1991.

\bibitem{AltschulerGrayson}
Steven~J. Altschuler and Matthew~A. Grayson.
\newblock {Shortening space curves and flow through singularities}.
\newblock {\em Journal of Differential Geometry}, 35(2):283 -- 298, 1992.

\bibitem{Andrews2010Mean}
Ben Andrews and Charles Baker.
\newblock Mean curvature flow of pinched submanifolds to spheres.
\newblock {\em Journal of Differential Geometry}, 85(3):357--395, 2010.

\bibitem{angenent1988zero}
Sigurd Angenent.
\newblock The zero set of a solution of a parabolic equation.
\newblock {\em Journal f{\"u}r die reine und angewandte Mathematik},
  390:79--96, 1988.

\bibitem{BourniLangfordLynch+2023+273+305}
Theodora Bourni, Mat Langford, and Stephen Lynch.
\newblock Collapsing and noncollapsing in convex ancient mean curvature flow.
\newblock {\em Journal für die reine und angewandte Mathematik (Crelles
  Journal)}, 2023(801):273--305, 2023.

\bibitem{blt_pancakes}
Theodora Bourni, Mat Langford, and Giuseppe Tinaglia.
\newblock Collapsing ancient solutions of mean curvature flow.
\newblock {\em J. Differential Geom.}, 119(2):187--219, 2021.

\bibitem{ChenYin2007_MCF_pseudolocality}
Bing-Long Chen and Le~Yin.
\newblock Uniqueness and pseudolocality theorems of the mean curvature flow.
\newblock {\em Communications in Analysis and Geometry}, 15(3):435--490, 2007.

\bibitem{choi2021convergence}
Beomjun Choi, Kyeongsu Choi, and Panagiota Daskalopoulos.
\newblock Convergence of curve shortening flow to translating soliton.
\newblock {\em American Journal of Mathematics}, 143(4):1043--1077, 2021.

\bibitem{ColdingMinicozzi2012}
Tobias~H. Colding and William P.~Minicozzi II.
\newblock Generic mean curvature flow i; generic singularities.
\newblock {\em Annals of Mathematics}, 175(2):755--833, 2012.

\bibitem{Colding2020Complexity}
Tobias~H. Colding and William P.~Minicozzi II.
\newblock Complexity of parabolic systems.
\newblock {\em Publications math{\'e}matiques de l'IH{\'E}S}, 132:83--135,
  2020.

\bibitem{Ding2020Minimal}
Qi~Ding.
\newblock Minimal cones and self-expanding solutions for mean curvature flows.
\newblock {\em Mathematische Annalen}, 376(1--2):359--405, 2020.

\bibitem{ecker1989mean}
Klaus Ecker and Gerhard Huisken.
\newblock Mean curvature evolution of entire graphs.
\newblock {\em Annals of Mathematics}, 130(3):453--471, 1989.

\bibitem{Grayson1987}
Matthew~A. Grayson.
\newblock The heat equation shrinks embedded plane curves to round points.
\newblock {\em Journal of Differential Geometry}, 26(2):285--314, 1987.

\bibitem{Httenschweiler2015CurveSF}
J{\"o}rg H{\"a}ttenschweiler.
\newblock Curve shortening flow in higher dimension.
\newblock 2015.

\bibitem{ladyzhenskaia1968linear}
Olga~Aleksandrovna Ladyzhenskaia, Vsevolod~Alekseevich Solonnikov, and Nina~N
  Ural'tseva.
\newblock {\em Linear and quasi-linear equations of parabolic type}, volume~23.
\newblock American Mathematical Soc., 1968.

\bibitem{litzinger2023singularities}
Florian Litzinger.
\newblock Singularities of low entropy high codimension curve shortening flow.
\newblock Preprint, arXiv:2304.02487.

\bibitem{lynch2020highcodimensionmeancurvature}
Stephen Lynch and Huy~The Nguyen.
\newblock High codimension mean curvature flow with surgery.
\newblock Preprint, arXiv:2004.07163.

\bibitem{Minarcik2020LongTerm}
Ji\v{r}\'{i} Minar\v{c}\'{i}k and Michal Bene\v{s}.
\newblock Long-term behavior of curve shortening flow in $\mathbb{R}^3$.
\newblock {\em SIAM Journal on Mathematical Analysis}, 52(2):1221--1231, 2020.

\bibitem{MP_flowtoroundpoint}
Alexander Mramor and Alexander Payne.
\newblock Nonconvex surfaces which flow to round points.
\newblock {\em Communications in Analysis and Geometry}, 32(3):837--887, 2024.

\bibitem{Naff2022}
Keaton Naff.
\newblock A planarity estimate for pinched solutions of mean curvature flow.
\newblock {\em Duke Math. J.}, 171(2):443--482, 2022.

\bibitem{naff2022singularity}
Keaton Naff.
\newblock Singularity models of pinched solutions of mean curvature flow in
  higher codimension.
\newblock {\em Journal f{\"u}r die reine und angewandte Mathematik (Crelles
  Journal)}, 2022(790):1--45, 2022.

\bibitem{nguyen2026highcodimensioncurveshortening}
Huy~The Nguyen and Artemis Vogiatzi.
\newblock High codimension curve shortening flow with free boundary.
\newblock Preprint, arXiv:2602.20865.

\bibitem{vogiatzi2023singularity}
Huy~The Nguyen and Artemis Vogiatzi.
\newblock Singularity models for high codimension mean curvature flow in
  riemannian manifolds.
\newblock Preprint, arXiv:2303.00414.

\bibitem{Polden1991EvolvingCurves}
Alexander Polden.
\newblock Evolving curves.
\newblock 1991.

\bibitem{Smoczyk2004_longtime}
Knut Smoczyk.
\newblock Longtime existence of the lagrangian mean curvature flow.
\newblock {\em Calculus of Variations and Partial Differential Equations},
  20(1):25--46, 2004.

\bibitem{stavrou1998selfsimilar}
Nikos Stavrou.
\newblock Selfsimilar solutions to the mean curvature flow.
\newblock {\em Journal f{\"u}r die reine und angewandte Mathematik},
  499:189--198, 1998.

\bibitem{QSun_2024}
Qi~Sun.
\newblock Curve shortening flow of space curves with convex projections.
\newblock Preprint, arXiv:2410.08399.

\bibitem{QSun_2025}
Qi~Sun.
\newblock Singularities of curve shortening flow with convex projections.
\newblock Preprint, arXiv:2510.14863.

\bibitem{Ton}
Yoshihiro Tonegawa.
\newblock Brakke's mean curvature flow: An introduction.
\newblock Springer Singapore, 2019.

\bibitem{Lu}
Lu~Wang.
\newblock Asymptotic structure of self-shrinkers.
\newblock Preprint, arXiv:1610.04904, 2016.

\bibitem{Wang2001Mean}
Mu-Tao Wang.
\newblock Mean curvature flow of surfaces in {Einstein} four-manifolds.
\newblock {\em Journal of Differential Geometry}, 57(2):301--338, 2001.

\bibitem{Wang2002_GraphicMCF}
Mu-Tao Wang.
\newblock Long-time existence and convergence of graphic mean curvature flow in
  arbitrary codimension.
\newblock {\em Inventiones Mathematicae}, 148(3):525--543, 2002.

\bibitem{Wang2004MCF}
Mu-Tao Wang.
\newblock The mean curvature flow smoothes lipschitz submanifolds.
\newblock {\em Communications in Analysis and Geometry}, 12(3):581--599, 2004.

\end{thebibliography}
\end{document}